\documentclass[12pt, reqno]{amsart}
\usepackage{amsmath,amssymb,amsthm,graphicx}
\usepackage[left=2.5cm,right=2.2cm,top=2.5cm,bottom=3.5cm]{geometry}
\usepackage{enumitem}

\usepackage[breaklinks]{hyperref}
\hypersetup{
	colorlinks = true, % Colours links instead of ugly boxes
	linkcolor = blue, % Colour of internal links
	citecolor = blue % Colour of citations
}
\usepackage{orcidlink}
\let\le\leqslant
\let\ge\geqslant
\let\a\alpha
\let\b\beta
\let\l\lambda

\let\eps\varepsilon
\let\mf\mathfrak
\let\mc\mathcal
\DeclareMathOperator{\ZZ}{\mathbb{Z}}
\DeclareMathOperator{\RR}{\mathbb{R}}
\DeclareMathOperator{\NN}{\mathbb{N}}
\DeclareMathOperator{\HH}{\mathbb{H}}
\DeclareMathOperator{\QQ}{\mathbb{Q}}
\DeclareMathOperator{\AI}{\mathbb{A}}

\DeclareMathOperator{\Isom}{Isom}
\DeclareMathOperator{\GL}{GL}
\DeclareMathOperator{\Vol}{Vol}
\DeclareMathOperator{\im}{Im}
\DeclareMathOperator{\Imm}{Im}
\DeclareMathOperator{\Ree}{Re}

\DeclareMathOperator{\Sal}{Sal}

\DeclareMathOperator{\OO}{O}

\theoremstyle{definition}
\newtheorem{lemma}{Lemma}[section]
\newtheorem{proposition}[lemma]{Proposition}
\newtheorem{definition}[lemma]{Definition}
\newtheorem{theorem}[lemma]{Theorem}
\newtheorem{remark}[lemma]{Remark}

\numberwithin{equation}{section}

\begin{document}
	
	\author{A\lowercase{lexandr} G\lowercase{rebennikov}}
	\title{M\lowercase{ultiplicities in the length spectrum and growth rate of }S\lowercase{alem numbers}}
	\address{A.~Grebennikov:  IMPA, Estrada Dona Castorina 110, Jardim Botanico, Rio de Janeiro, Brazil}
	\email{sagresash@yandex.ru}	
	\thanks{Alexandr Grebennikov \orcidlink{0000-0003-4715-8588}: IMPA, Rio de Janeiro, Brazil; sagresash@yandex.ru}
	
	\begin{abstract}
		
		We prove that mean multiplicities in the length spectrum of a non-compact arithmetic hyperbolic orbifold of dimension $n \geqslant 4$ have exponential growth rate		
		$$
		\langle g(L) \rangle \geqslant c \frac{e^{([n/2] - 1)L}}{L^{1 + \delta_{5, 7}(n) }},
		$$
		extending the analogous result for even dimensions of Belolipetsky, Lalín, Murillo and Thompson. Our proof is based on the study of (square-rootable) Salem numbers. 
		
		As a counterpart, we also prove an asymptotic formula for the distribution of square-rootable Salem numbers by adapting the argument of G\"otze and Gusakova. It shows that one can not obtain a better estimate on mean multiplicities using our approach. 
		
	\end{abstract}
	
	\maketitle	
	
	\section{Introduction}
	
	The length spectrum of a (negatively curved) Riemannian manifold is defined as the list of lengths of its primitive closed geodesics written in ascending order and with multiplicities. 
	The length spectrum has received a lot of attention because it encodes geometric properties of the manifold and is closely related to the eigenvalue spectrum of its Laplace-Beltrami operator.
	
	The multiplicities in the length spectrum are believed to reflect how rigid our space is. A generic manifold of variable negative curvature is not supposed to have any additional structure, thus its spectrum has no non-trivial multiplicities \cite{Anosov83}. Non-arithmetic hyperbolic manifolds represent intermediate situation: their multiplicities are known to be unbounded in dimensions two \cite{Randol80} and three \cite{Masters00}, and mean multiplicities can sometimes grow exponentially \cite{BS04}, though with a small constant in the exponent. 
	
	Current note is devoted to ``the structured'' case: arithmetic hyperbolic orbifolds, which may be defined as quotients of the hyperbolic space $\HH^n$ by an action of an arithmetic lattice $\Gamma \subset \Isom(\HH^n)$. Our Theorem \ref{thm:multiplicities_geom} states that mean multiplicities in the length spectrum of a non-compact arithmetic hyperbolic orbifold of dimension $n \ge 4$ have exponential growth rate. Analogous statements in dimensions two \cite{BGGS92}, \cite{Bolte93} and three \cite{Marklof96} are well-known, in these cases even sharper bounds are known and non-compactness is not required. The case of even dimension at least four is due to Belolipetsky, Lalín, Murillo and Thompson \cite[Proposition 3]{Bel21}, we adapt their approach to work in odd dimensions as well and present a uniform proof for both cases for the sake of transparency.
	
	\begin{theorem}
		\label{thm:multiplicities_geom}
		Let $\mc{O} = \HH^n/\Gamma$ be a non-compact arithmetic hyperbolic orbifold of dimension $n \ge 4$, and let $G(L)$ be the set of its primitive closed geodesics of length at most $L$. Then for sufficiently large $L \ge L_0(n, \Gamma)$ we have
		\[
		\langle g(L) \rangle = \frac{\# G(L)}{\# \{\ell(\gamma) : \gamma \in G(L)\}} \ge c \frac{e^{([n/2] - 1)L}}{L^{1 + \delta_{5, 7}(n) }},
		\]
		where $\delta_{5, 7}(n) = \begin{cases} 1, & \text{if } n = 5, 7; \\ 0, & \text{otherwise,} \end{cases}$ \;\;\; and
		\begin{enumerate}
			\item[(1)] when $\Gamma$ is torsion-free (or, equivalently, when $\mc{O}$ is a manifold), $c = c(n) > 0$ is a constant depending only on the dimension;
			
			\item[(2)] in the general case, $c = c(n, \Gamma) > 0$ is a constant depending only on the lattice and its dimension.
		\end{enumerate}
	\end{theorem} 
	
	It will be convenient for us to reformulate Theorem \ref{thm:multiplicities_geom} in more algebraic terms. An isometry $H \in \Isom(\HH^n) \simeq \OO^+(n, 1)$ is called a \textit{hyperbolic transformation} if there is a unique geodesic in $\HH^n$ along which $H$ acts as a translation by distance $\ell(H) > 0$. 
	This geodesic is called the \textit{axis} of $H$, and $\ell(H)$ is called the \textit{translation length} of $H$.	
	
	When $\Gamma$ is torsion-free, there is a well-known one-to-one correspondence between the set of primitive closed geodesics in $\HH^n/\Gamma$ and the set of conjugacy classes of primitive hyperbolic transformations in $\Gamma$ (for instance, see \cite[Theorem 9.6.2]{Ratcliffe19}). Moreover, the length $\ell(\gamma)$ of any closed geodesic $\gamma$ coincides with the translation length $\ell(H)$ of any hyperbolic transformation $H$ in the corresponding conjugacy class. 
	
	If $\Gamma$ contains non-trivial elements of finite order, this bijective correspondence slightly fails: different conjugacy classes may correspond to the same closed geodesic in the quotient. However, one can check that the number of conjugacy classes corresponding to the same geodesic is bounded above by some constant depending only on the lattice $\Gamma$. This issue was overlooked in the first version of this work, and explains the reason for introducing two cases in the statement of Theorem \ref{thm:multiplicities_geom}. We refer the reader to Section \ref{sec:bijection} for a more detailed discussion of this correspondence and its use in this note.
	
	\medskip
	
	Therefore, Theorem \ref{thm:multiplicities_geom} follows from Theorem \ref{thm:multiplicities_alg}, stated in terms of the lattice $\Gamma$.
	
	\begin{theorem}
		\label{thm:multiplicities_alg}
		Let $\Gamma$ be a non-cocompact arithmetic lattice in $\Isom(\HH^n) \simeq \OO^+(n, 1)$ for $n \ge 4$. Consider the set $\Gamma_h(L)$ of conjugacy classes of primitive hyperbolic transformations $H \in \Gamma$ with translation lengths bounded as $\ell(H) \le L$. Group the elements of this set into parts according to the value of $\ell(H)$. Then for sufficiently large $L \ge L_0(n, \Gamma)$ average size of such part can be bounded below as
		\[
		\frac{\# \Gamma_h(L)}{\#\{ \ell(H) : H \in \Gamma_h(L) \}} \ge c  \frac{e^{([n/2]-1)L}}{L^{1 + \delta_{5, 7}(n)}}
		\]
		for some constant $c = c(n) > 0$.
	\end{theorem}
	
	Consider a hyperbolic transformation $H$ as a matrix in $\OO^+(n, 1)$. It has eigenvalues $e^{\ell(H)}$ and $e^{-\ell(H)}$, while other its eigenvalues are complex numbers of absolute value one. When $H$ comes from an arithmetic lattice $\Gamma$, one may expect its characteristic polynomial to have, in a sense, ``integral'' coefficients. Thus it is not surprising that our proof of Theorem \ref{thm:multiplicities_alg} makes use of the number-theoretic concept of Salem numbers: a \textit{Salem number} is a real algebraic integer $\l > 1$, such that it is conjugate to $\l^{-1}$ and its remaining conjugates lie on the unit circle. Note that we allow Salem numbers to have degree $2$, to be consistent with \cite{Bel21} and \cite{ERT15}. 
	
	We also need a more technical notion, which was introduced by Emery, Ratcliffe and Tschantz.
	
	\begin{definition}[{\cite[Section 1.3]{ERT15}}]
		\label{def:square-rootable}
		A Salem number $\l$ is called \textit{square-rootable} over a field $K \subset \QQ(\l + \l^{-1})$ if there exist a totally positive $\a \in K$ and a monic palindromic polynomial $q$, such that even degree coefficients of $q$ lie in $K$, odd degree coefficients of $q$ lie in $\sqrt{\a}K$, and
		\[
		q(x)q(-x) = p(x^2), 
		\]
		where $p$ is the minimal polynomial of $\l$ over $K$.
	\end{definition}
	
	The following result of Emery, Ratcliffe and Tschantz \cite{ERT15} formalizes the relation between (square-rootable) Salem numbers and the length spectrum of arithmetic hyperbolic orbifolds. 
	
	\begin{theorem}[{\cite[Theorem 1.1 and Theorem 1.6]{ERT15}}]
		\label{thm:ERT}	
		Suppose $\Gamma \subset \Isom(\HH^n)$ is an arithmetic lattice of the simplest type defined over a totally real number field $K$, $H \in \Gamma$ is a hyperbolic transformation, and $\ell(H)$ is its translation length. If $n$ is even, then $\l_1 = e^{\ell(H)}$ is a Salem number, with $\deg_K(\l_1) \le n+1$. If $n$ is odd, then $\l_2 = e^{2\ell(H)}$ is a Salem number, square-rootable over $K$, with $\deg_K(\l_2) \le n+1$.
	\end{theorem}
	
	Since in our setting the lattice $\Gamma$ is non-cocompact, it is thus of the simplest type defined over $K = \QQ$ (see \cite[Sections 1-2]{LM93}).	
	
	\medskip
	
	As Theorem \ref{thm:ERT} suggests, estimates on the order of growth of (square-rootable) Salem numbers are an important ingredient in our proof of Theorem \ref{thm:multiplicities_alg}. More precisely, our argument relies on the following two propositions.
	
	\begin{proposition}
		\label{prop:salem_all}
		Denote by $\Sal_m(Q)$ the set of Salem numbers of degree $2m$ lying in the interval $(1, Q]$.
		Then for a fixed $m \in \NN$ and $Q \to \infty$,
		\[
		\#\Sal_m(Q) = O(Q^m).
		\]
	\end{proposition}	
	
	\begin{proposition}
		\label{prop:salem_square-rootable}
		Denote by $\Sal_m^{sq}(Q)$ the set of square-rootable (over $\QQ$) Salem numbers of degree $2m$ lying in the interval $(1, Q]$. 
		Then for a fixed $m \in \NN$ and $Q \to \infty$,
		\[
		\#\Sal_m^{sq}(Q) = O(Q^{m/2} \cdot \eta(Q, m)), \text{ where }
		\]
		\[
		\eta(Q, m) = \begin{cases} Q^{1/2}, & \text{ if } m = 1, 2; \\ \log(Q), & \text{ if } m = 3, 4; \\
			1, & \text{ if } m \ge 5. \end{cases}			
		\]
	\end{proposition}
	
	Proposition \ref{prop:salem_all} is not novel: theorem of G\"otze and Gusakova \cite[Theorem 1.1]{GG19} even provides an asymptotic formula for the number of all Salem numbers of given degree on the initial segment of the real line
	\begin{equation}
		\label{eq:GG-introduction}
		\#\Sal_m(Q) = w_{m-1}Q^m + O(Q^{m-1}).
	\end{equation}
	Proposition \ref{prop:salem_square-rootable} is new, but it can be also deduced from our more technical Theorem \ref{thm:asymptotics}, which is discussed later in the introduction. Nevertheless, in Section \ref{sec:proof_main} we present short and elementary proofs of these propositions in order to keep our proof of Theorem \ref{thm:multiplicities_alg} more simple and self-contained.
	
	\medskip
	
	The asymptotic formula \eqref{eq:GG-introduction} of G\"otze and Gusakova also implies that the bound from Proposition \ref{prop:salem_all} is essentially optimal. It turns out that a similar statement holds for Proposition \ref{prop:salem_square-rootable} as well. It is not hard to check that every Salem number of degree $2$ is square-rootable (over $\QQ$), and thus
	\[
	\#\Sal_1^{sq}(Q) = Q + O(1).
	\]
	Belolipetsky, Lalín, Murillo and Thompson \cite[Theorem 2]{Bel21} obtained an asymptotic formula for square-rootable (over $\QQ$) Salem numbers of degree $4$
	\[
	\#\Sal_2^{sq}(Q) = \frac{4}{3} Q^{3/2} + O(Q).
	\]		
	We extend their result to larger degrees, modifying the approach of G\"otze and Gusakova. Namely, in Section \ref{sec:GG-modification} we prove the following Theorem \ref{thm:asymptotics}, which determines the cardinality of the set $\Sal_m^{sq}(Q)$ up to $1 + o(1)$ factor when $m$ is odd, and up to a multiplicative constant factor when $m$ is even.
	
	%-------------new Version of second theorem without \err----------------------
	\begin{theorem}		
		\label{thm:asymptotics}
		Denote by $\Sal_m^{sq}(Q)$ the set of square-rootable (over $\QQ$) Salem numbers of degree $2m$ lying in the interval $(1, Q]$. Then for fixed $m \ge 3$ and $Q \to \infty$,
		\begin{enumerate}			
			\item[(1)] if $m = 3$, then
			\[
			\#\Sal_m^{sq}(Q) = w_{m-1} \frac{6}{\pi^2} Q^{m/2} \log(Q) + O(Q^{m/2}); 
			\]
			\item[(2)] if $m = 4$, then
			%			\[
			%			 \frac{1}{2^{2m}} \cdot w_{m-1} \frac{6}{\pi^2} Q^{m/2} \log(Q) + O(Q^{m/2}) \le \#\Sal_m^{sq}(Q) \le w_{m-1} \frac{6}{\pi^2} Q^{m/2} \log(Q) + O(Q^{m/2});
			%			\]
			\[
			\#\Sal_m^{sq}(Q) \ge \frac{1}{2^{2m}} \cdot w_{m-1} \frac{6}{\pi^2} Q^{m/2} \log(Q) + O(Q^{m/2}),
			\]
			\[
			\#\Sal_m^{sq}(Q) \le w_{m-1} \frac{6}{\pi^2} Q^{m/2} \log(Q) + O(Q^{m/2});
			\]
			\item[(3)] if $m \ge 5$ and $m$ is odd, then
			%			\[
			%			\#\Sal_m^{sq}(Q) = w_{m-1} \frac{\zeta(\frac12 [\frac{m+1}{2}])}{\zeta([\frac{m+1}{2}])} Q^{m/2} + O(Q^{(m-1)/2} \log Q);
			%			\]
			\[
			\#\Sal_m^{sq}(Q) = w_{m-1} \frac{\zeta(
				\frac{m+1}{4})}{\zeta(\frac{m+1}{2})} Q^{m/2} + O(Q^{(m-1)/2} \log Q);
			\]
			\item[(4)] if $m \ge 5$ and $m$ is even, then
			%			\[
			%			\#\Sal_m^{sq}(Q) \ge \frac{1}{2^{2m}} \cdot w_{m-1} \frac{\zeta(\frac12 [\frac{m+1}{2}])}{\zeta([\frac{m+1}{2}])} Q^{m/2}  + O(Q^{(m-1)/2} \log Q),
			%			\]
			%			\[
			%			\#\Sal_m^{sq}(Q) \le w_{m-1} \frac{\zeta(\frac12 [\frac{m+1}{2}])}{\zeta([\frac{m+1}{2}])} Q^{m/2}  + O(Q^{(m-1)/2} \log Q). 
			%			\]
			\[
			\#\Sal_m^{sq}(Q) \ge \frac{1}{2^{2m}} \cdot w_{m-1} \frac{\zeta(\frac{m}{4})}{\zeta(\frac{m}{2})} Q^{m/2}  + O(Q^{(m-1)/2} \log Q),
			\]
			\[
			\#\Sal_m^{sq}(Q) \le w_{m-1} \frac{\zeta(\frac{m}{4})}{\zeta(\frac{m}{2})} Q^{m/2}  + O(Q^{(m-1)/2} \log Q). 
			\]
		\end{enumerate}
		Here $w_m$ are explicit constants defined by the formula		
		\begin{equation}
			\label{eq:w_m}
			w_m = \frac{2^{m(m+1)}}{m+1}\prod_{k = 0}^{m-1} \frac{(k!)^2}{(2k+1)!},
		\end{equation}
		and $\zeta$ is the Riemann zeta function.		
	\end{theorem}
	%---------------------------------------------------------

	%-------------old Version of second theorem with \err----------------------
	%	\begin{theorem}		
		%	\label{thm:asymptotics}
		%		Denote by $\Sal_m^{sq}(Q)$ the set of square-rootable (over $\QQ$) Salem numbers of degree $2m$ lying in the interval $(1, Q]$. Then for a fixed $m \ge 3$ and $Q \to \infty$
		%		\begin{enumerate}			
			%			\item if $m = 3, 4$, then
			%			\[
			%				\#\Sal_m^{sq}(Q) = w_{m-1} \frac{6}{\pi^2} Q^{m/2} \log(Q) \cdot \err(Q), 
			%			\]
			%			\item if $m \ge 5$, then
			%			\[
			%				\#\Sal_m^{sq}(Q) = w_{m-1} \frac{\zeta(\frac12 [\frac{m+1}{2}])}{\zeta([\frac{m+1}{2}])} Q^{m/2} \cdot \err(Q), 
			%			\]
			%		\end{enumerate}
		%		where the error factor satisfies
		%		\[
		%			%\err(R) = \begin{cases} 1 + o(1), & \text{ when $m$ is odd; } \\ \Theta(1), & \text{ when $m$ is even. }\end{cases}
		%			\begin{cases} \err(Q) = 1 + o(1), & \text{ when $m$ is odd; } \\ \eps + o(1) \le \err(Q) \le 1 + o(1), & \text{ when $m$ is even. }\end{cases}
		%		\]
		%		%Here $w_m$ are explicit constants defined by formula \eqref{eq:w_m}, $\eps = \eps(m) > 0$ is an implicit constant, and $\zeta$ is the Riemann zeta function.
		%		Here $w_m$ are explicit constants defined by the formula		
		%		\begin{equation}
			%			\label{eq:w_m}
			%			w_m = \frac{2^{m(m+1)}}{m+1}\prod_{k = 0}^{m-1} \frac{(k!)^2}{(2k+1)!},
			%		\end{equation}
		%		$\eps = \eps(m) > 0$ is an implicit constant, and $\zeta$ is the Riemann zeta function.
		%		
		%	\end{theorem}
	%---------------------------------------------------------
	
	Theorem \ref{thm:asymptotics} is a purely number-theoretic statement. From geometric perspective, the main motivation to prove it is to ensure that Proposition \ref{prop:salem_square-rootable} is essentially tight, and therefore one can not obtain a better bound in Theorem \ref{thm:multiplicities_alg} with our approach. Also this theorem may be used to write down an explicit expression for the constant $c(n)$ in Theorem \ref{thm:multiplicities_alg}, see Remark \ref{remark:explicit_constant}.
	
	\bigskip
	
	\noindent\textbf{Acknowledgements.} I am grateful to my advisor Mikhail Belolipetsky for bringing my attention to this problem and helpful comments that have significantly improved this note. I am also thankful to the anonymous referees for their careful reading and valuable suggestions. 
	This study was financed in part by the Coordenação de Aperfeiçoamento de Pessoal de Nível Superior, Brasil (CAPES).
		
	\section{Notation and preliminaries}
	
	\subsection{Salem numbers} A polynomial $p$ is called palindromic if it satisfies $p(x) = x^{\deg p} p(x^{-1})$. It is not hard to prove that the minimal polynomial of a Salem number has to be palindromic of even degree. See \cite{Smyth15} for a general survey on Salem numbers.
	
	\medskip
	
	\noindent Notation $\AI$ is used for the ring of algebraic integers over $\QQ$. We denote by $\Sal_m \subset \AI$ the set of Salem numbers of degree $2m$, and let
	\[
	\Sal_m(Q) = \Sal_m \;\cap\; (1, Q].
	\]
	Similarly, we denote by $\Sal^{sq}_m \subset \AI$ the set of square-rootable (over $\QQ$) Salem numbers of degree $2m$, and let
	\[
	\Sal_m^{sq}(Q) = \Sal_m^{sq} \;\cap\; (1, Q].
	\]
	
	\subsection{Asymptotic behavior of functions} For real-valued functions $f(x)$ and $g(x)$ we use the following standard notation to compare their asymptotic behavior as $x \to \infty$.
	
	\begin{itemize}
		\item $f(x) = O(g(x))$ if $|f(x)| \le C g(x)$ for some constant $C > 0$;
		\item $f(x) = o(g(x))$ if for any $\eps > 0$ there exists $N = N(\eps)$ such that $|f(x)| \le \eps g(x)$ for $x \ge N$;
		\item $f(x) \sim g(x)$ if $f(x) - g(x) = o(g(x))$.
	\end{itemize}
	Here we emphasize that implicit constants $C$ and $N = N(\eps)$ \textbf{are allowed} to depend additionally on the dimension of the manifold (usually denoted by $n$) or degree of the Salem number (usually denoted by $2m$), but not on anything else.
	
	\medskip
	
	For $x \in \RR$ we use $\log (x)$ for the natural (base $e$) logarithm of $x$, and $[x]$ for the largest integer not exceeding $x$. We also write $\# X$ for the cardinality of a set $X$. 
	
	\subsection{Arithmetic orbifolds of the simplest type} To give the definitions, we follow the exposition of \cite[Section 2.3]{ERT15} (see also \cite[Section 12.8]{Ratcliffe19} for more details). A quadratic form $f$ in $n+1$ variables over a totally real number field $K$ is called \textit{admissible} if it has signature $(n, 1)$, and for each non-trivial embedding $\sigma: K \to \RR$ corresponding quadratic form $f^\sigma$ over $\sigma(K)$ is positive definite. For a subring $R \subset \RR$ denote by $\OO(f, R)$ the group of matrices preserving this quadratic form
	\[
	\OO(f, R) = \{T \in \GL(n+1, R) \mid f(Tx) = f(x) \text{ for all } x \in \RR^{n+1}\},
	\]
	and let $\OO'(f, R)$ be its subgroup, also keeping both connected components of the negative cone $\{ x \in \RR^{n+1} \mid f(x) < 0\}$ invariant. 
	
	Let $f_n$ be the standard Lorentzian form in $n+1$ variables
	\[
	f_n(x) = x_1^2 + \ldots + x_n^2 - x_{n+1}^2.
	\]	
	A subgroup $\Gamma \subset \Isom(\HH^n)$ is called an \textit{arithmetic subgroup of the simplest type defined over $K$} if there exist $M \in \GL(n+1, \RR)$ and an admissible quadratic form $f$ over $K$ in $n+1$ variables such that
	\begin{itemize}
		\item $f(Mx) = f_n(x)$ for any $x \in \RR^{n+1}$;
		\item subgroups $M \Gamma M^{-1}$ and $\OO'(f, \mathfrak{o}_K)$ of $\OO'(f, \RR)$, where $\mf{o}_K$ is the ring of integers of $K$, are commensurable (that is, their intersection has finite index in both of them).
	\end{itemize}
	
	Such $\Gamma$ is always a discrete subgroup of $\Isom(\HH^n)$ of finite covolume. The hyperbolic orbifold $\HH^n/\Gamma$ in this case is also called \textit{arithmetic of the simplest type defined over $K$}.	
	
	As mentioned in the introduction, for the case of non-cocompact arithmetic lattices this construction is actually exhaustive.
	
	\begin{proposition}[{\cite[Sections 1-2]{LM93}}]
		\label{prop:K=Q}
		Any non-cocompact arithmetic lattice in $\Isom(\HH^n)$ is of the simplest type defined over $\QQ$.
	\end{proposition}
	
	This allows to carry out our analysis only for the case $K = \QQ$. In particular, for the remaining part of the note square-rootable Salem numbers are assumed to be square-rootable over $\QQ$ in the sense of Definition \ref{def:square-rootable}.

	\subsection{Closed geodesics in orbifolds, and their relation to conjugacy classes} \label{sec:bijection}
	Let $\Gamma$ be a lattice in $\Isom(\HH^n)$. A map $\gamma: \RR \to \mc{O} = \HH^n/\Gamma$ is called a \textit{closed geodesic of length $L$}
	\footnote{Formally, a closed geodesic is a pair $([\gamma], L)$, where $[\gamma]$ is the equivalence class of $\gamma$ up to shifts of argument, and $L$ is its period. In other words, a closed geodesic ``remembers'' its length and orientation, but not its starting point.}
	if it is periodic with period $L$ and lifts to an infinite geodesic $\tilde \gamma: \RR \to \HH^n$, satisfying
	\[
	\pi(\tilde \gamma(t)) = \gamma(t), \text{ where $\pi$ is the projection $\HH^n \to \HH^n/\Gamma$. }
	\]
	When $\mc{O}$ is a manifold (in other words, when $\Gamma$ is torsion-free), this definition coincides with the classical definition of a geodesic as a local isometry. 
	In general, however, a closed geodesic can fail to be locally distance-minimizing around singular points of the orbifold (hence some literature uses the term \textit{orbifold geodesic} for this notion).
	
	\medskip
	
	Now we define and analyze the map $\mc{F}$ from the set of hyperbolic conjugacy classes in $\Gamma$ to the set of closed geodesics in $\HH^n/\Gamma$, adapting the proof of \cite[Theorem 9.6.2]{Ratcliffe19} to the setting of orbifolds.	While the content of this section seems to be well-known, we were not able to find an explicit reference to it.
	
	\medskip
	
	Let $\b$ be the axis of a hyperbolic transformation $H \in \Gamma$. Define $\mc{F}([H])$ as projection of this axis to the quotient: the closed geodesic $\gamma: t \mapsto \pi(\b(t))$ of length $\ell(H)$.
	Since for any $g \in \Gamma$ the axis $g \b$ of hyperbolic transformation $g H g^{-1}$ projects to the same closed geodesic, this map is well-defined. Next we verify the following properties of the map $\mc{F}$:
	
	\begin{enumerate}
		\item[1.] $\mc{F}$ is surjective;
		
		\medskip
		
		Consider a closed geodesic $\gamma$ of length $L$, and its lift $\tilde \gamma$. Since $\gamma(0) = \gamma(L)$, there exists an element $g_0 \in \Gamma$, such that $\tilde \gamma(0) = g_0 \tilde \gamma(L)$. Define another geodesic in $\HH^n$ as
		\[
		\tilde \gamma'(t) = g_0 \tilde \gamma(t + L) \text{ for any } t \in \RR.
		\]
		It also projects to $\gamma$ under $\pi$, and agrees with $\tilde \gamma$ at $t = 0$: $\tilde \gamma(0) = \tilde \gamma'(0) = x$.
		Then a version of the unique lifting property \cite[Theorem 13.1.6]{Ratcliffe19} implies that $\tilde \gamma' = u \tilde \gamma$ for some elliptic transformation
		\footnote{An isometry $u$ of $\HH^n$ is called an \textit{elliptic transformation} if its action on $\HH^n$ has a fixed point. When $u$ lies in a discrete subgroup of $\Isom(\HH^n)$, it is elliptic if and only if it has finite order (see \cite[Section 5.5]{Ratcliffe19}).} 
		$u \in \Gamma$, fixing the point $x$. Thus 
		\[
		u \tilde \gamma(t) = g_0 \tilde \gamma(t + L),
		\]
		and $g = g_0^{-1} u \in \Gamma$ is a hyperbolic transformation acting on $\tilde \gamma$ as a translation by $L$. Therefore, $\mc{F}(g) = \gamma$.
		
		\medskip
		
		\item[2.] If $\Gamma$ is torsion-free, then $\mc{F}$ is injective;
		
		\medskip
		
		Consider two hyperbolic transformations $H_1, H_2 \in \Gamma$, and suppose that their axes $\gamma_1, \gamma_2$ project to the same closed geodesic of length $L$ in the quotient $\HH^n/\Gamma$. Then $\gamma_1(0) = g \gamma_2(0)$ for some $g \in \Gamma$. Since $\Gamma$ is torsion-free, we can use the usual unique lifting property of geodesics to conclude that $\gamma_1 = g \gamma_2$. Then both $H_1$ and $g H_2 g^{-1}$ act on the geodesic $\gamma_1$ as translations by $L$. Therefore, $s = H_1^{-1} gH_2g^{-1} \in \Gamma$ is an elliptic element of $\Gamma$, which fixes each point of this geodesic. But $\Gamma$ is torsion-free, thus $s = 1$, and $H_1$ is conjugate to $H_2$.
		
		\medskip
		
		\item[3.] In general, $\# \mc{F}^{-1}(\gamma) \le C(\Gamma)$ for any closed geodesic $\gamma$;
		
		\medskip
		
		Similarly, consider two hyperbolic transformations $H_1, H_2 \in \Gamma$, and suppose that their axes $\gamma_1, \gamma_2$ project to the same closed geodesic of length $L$ in the quotient $\HH^n/\Gamma$. Then $\gamma_1(0) = g_0 \gamma_2(0)$ for some $g_0 \in \Gamma$. By a version of the unique lifting property \cite[Theorem 13.1.6]{Ratcliffe19}, we conclude that $\gamma_1 = u g_0 \gamma_2$ for some elliptic transformation $u \in \Gamma$. 
		
		So, $\gamma_1 = g \gamma_2$ for $g = u g_0 \in \Gamma$. Then both $H_1$ and $g H_2 g^{-1}$ act on the geodesic $\gamma_1$ as translations by $L$. Therefore, $s = H_1^{-1} gH_2g^{-1}$ fixes each point of this geodesic. So, $H_2$ lies in the conjugacy class of $H_1 s$ for some $s \in S$, where $S$ is the subgroup of $\Gamma$, fixing the geodesic $\gamma_1$.
		
		By Selberg's lemma $\Gamma$ has a torsion-free normal subgroup $\Gamma_0$ of finite index. Then $\Gamma_0$ acts on $\HH^n$ freely, thus the projection map $S \to \Gamma/\Gamma_0$ has to be injective, and the size of $S$ has to be bounded above by $[\Gamma:\Gamma_0]$. So, preimage of each closed geodesic with respect to $\mc{F}$ contains at most $C(\Gamma) = [\Gamma:\Gamma_0]$ distinct conjugacy classes.		
	\end{enumerate}
	
	\medskip
	
	When speaking about length spectrum, we consider only primitive closed geodesics: a closed geodesic $\gamma: \RR \to \HH^n/\Gamma$ of length $L$ is called \textit{primitive} if $L$ is its minimal period. The map $\mc{F}$ relates this notion to its ``algebraic counterpart'', conjugacy classes of primitive hyperbolic transformations: a hyperbolic transformation $H \in \Gamma$ is called \textit{primitive} if it cannot be represented as $(H')^k$ for another hyperbolic transformation $H' \in \Gamma$ and $k \ge 2$ (since replacing $H$ by its conjugate in $\Gamma$ clearly preserves primitivity, we can also talk about \textit{primitive conjugacy classes}).
	Namely, we check that a closed geodesic $\gamma$ is primitive if and only if every conjugacy class of hyperbolic transformations in its preimage $\mc{F}^{-1}(\gamma)$ is primitive. 
	
	\medskip
	
	\begin{enumerate}
		\item[4.] If a hyperbolic transformation $H$ is not primitive, then the closed geodesic $\mc{F}([H])$ is not primitive;
		
		\medskip
		
		If $H = (H')^k$, then the axis of $H'$ is also the axis of $H$, and the translation length $\ell(H)$ is equal to $k \cdot \ell(H')$. Thus projection of the axis of $H$ to $\HH^n/\Gamma$ is periodic with period $\ell(H') = \frac{\ell(H)}{k} < \ell(H)$.
		
		\medskip
		
		\item[5.] If a closed geodesic $\gamma$ is not primitive, then at least one hyperbolic conjugacy class in the preimage $\mc{F}^{-1}(\gamma)$ is not primitive.
		
		\medskip
		
		If $\gamma$ has length $L$ and is not primitive, then it is also periodic with a smaller period $L' = \frac{L}{k} < L$. Let $\tilde \gamma$ be a lift of $\gamma$ to an infinite geodesic in $\HH^n$. As in the proof of property 1, we can construct a hyperbolic transformation $H' \in \Gamma$, which acts on the lift $\tilde \gamma$ as a translation by $L'$. Then $(H')^k$ acts on $\tilde \gamma$ as a translation by $L$. So, $\mc{F}([(H')^k]) = \gamma$, and it is clearly not primitive.		
	\end{enumerate}
	
	Finally, we use the properties of map $\mc{F}$ to deduce Theorem \ref{thm:multiplicities_geom} from Theorem \ref{thm:multiplicities_alg}. Recall that we denote by $G(L)$ the set of primitive closed geodesics in $\HH^n/\Gamma$ with lengths at most $L$, and we denote by $\Gamma_h(L)$ the set of conjugacy classes of primitive hyperbolic transformations in $\Gamma$ with translation lengths at most $L$.
	
	\begin{enumerate}
		\item[(1)] If $\Gamma$ is torsion-free, then properties 1, 2, 4, 5 together imply that
		$\# G(L) = \# \Gamma_h(L)$, which allows us to deduce Theorem \ref{thm:multiplicities_geom}(1) from Theorem \ref{thm:multiplicities_alg}.
		
		\item[(2)] Denote by $G'(L)$ the set of \textbf{not} primitive closed geodesics in $\HH^n/\Gamma$ with lengths at most $L$, and denote by $\Gamma'_h(L)$ the set of conjugacy classes of \textbf{not} primitive hyperbolic transformations in $\Gamma$ with translation lengths at most $L$. Then by properties 1 and 3
		\[
		\# G(L) + \# G'(L) \ge \frac{\#\Gamma_h(L) + \#\Gamma'_h(L)}{C(\Gamma)},
		\]
		and by property 5
		\[
		\# G'(L) \le \# \Gamma'_h(L).
		\]
		Results of Gangolli and Warner \cite[Proposition 5.4]{GW80} imply that most of hyperbolic conjugacy classes are primitive: 
		$\frac{\# \Gamma'_h(L)}{\# \Gamma_h(L)} \to 0$ when $L \to \infty$ and $\Gamma$ is fixed. Therefore, for sufficiently large $L \ge L_0(n, \Gamma)$
		\begin{equation}
			\label{eq:2.4-1}
			\# G(L) \ge \frac{\#\Gamma_h(L) + \#\Gamma'_h(L)}{C(\Gamma)} - \# \Gamma'_h(L) \ge \frac{\#\Gamma_h(L)}{2 \cdot C(\Gamma)}.
		\end{equation}
		Properties 1 and 4 give
		\[
		\{\ell(\gamma) : \gamma \in G(L)\} \subset \{\ell(H) : H \in \Gamma_h(L)\},
		\]
		and we combine it with the estimate \eqref{eq:2.4-1} to deduce Theorem \ref{thm:multiplicities_geom}(2) from Theorem \ref{thm:multiplicities_alg}.
		
	\end{enumerate}
		
	\section{Proof of Theorem \ref{thm:multiplicities_alg}}	\label{sec:proof_main}	
	
	%	\begin{proposition}
		%	\label{prop:salem_all}
		%	For a fixed $m \in \NN$ and $Q \to \infty$,
		%	\[
		%		\#\Sal_m(Q) = O(Q^m).
		%	\]
		%	\end{proposition}
	
	First we present a short proof Proposition \ref{prop:salem_all}, which states that 
	\[
	\#\Sal_m(Q) = O(Q^m).
	\]
	\begin{proof}[\textbf{Proof of Proposition \ref{prop:salem_all}}]
		It suffices to estimate the number of monic palindromic polynomials $p \in \ZZ[x]$ of degree $2m$, such that all its roots except one have absolute value at most $1$, and the remaining one has absolute value at most $Q$. Suppose that
		$$
		p(x) = x^{2m} + p_{2m-1} x^{2m-1} + \ldots + p_1 x + 1, \; p_k = p_{2m-k}
		$$
		is such polynomial. We express it via its roots as
		$$
		p(x) = (x - \b_1)(x - \b_2) \ldots (x - \b_{2m}),  \;\;\; |\b_1| \le Q, \; |\b_i| \le 1 \text{ for } 2 \le i \le 2m,
		$$
		to conclude that
		$$
		|p_{2m-k}| \le \sum_{1 \le i_1 < \ldots < i_k \le 2m} |\b_{i_1} \ldots \b_{i_k}| \le {2m \choose k} \cdot Q = O(Q).
		$$
		So, we have $O(Q)$ options for each of the coefficients $p_1, \ldots, p_m$. Therefore, there are $O(Q^m)$ such polynomials.		
	\end{proof}
	
	%	\begin{proposition}
		%	\label{prop:salem_square-rootable}		
		%		For a fixed $m \in \NN$ and $Q \to \infty$
		%		\[
		%			\#\Sal_m^{sq}(Q) = O(Q^{m/2} \cdot \eta(Q, m)), \text{ where }
		%		\]
		%		\begin{equation}
			%			\label{eq:eta}
			%			\eta(Q, m) = \begin{cases} Q^{1/2}, & \text{ if } m = 1, 2; \\ \log(Q), & \text{ if } m = 3, 4; \\
				%				1, & \text{ if } m \ge 5. \end{cases}			
			%		\end{equation}
		%	\end{proposition}
	
	Next we use a similar argument to prove Proposition \ref{prop:salem_square-rootable}, which states that	
	\[
	\#\Sal_m^{sq}(Q) = O(Q^{m/2} \cdot \eta(Q, m)), \text{ where }
	\]
	\[
	\eta(Q, m) = \begin{cases} Q^{1/2}, & \text{ if } m = 1, 2; \\ \log(Q), & \text{ if } m = 3, 4; \\
		1, & \text{ if } m \ge 5. \end{cases}
	\]	
	
	\begin{proof}[\textbf{Proof of Proposition \ref{prop:salem_square-rootable}}]
		If $\l \le Q$ is a square-rootable Salem number with the minimal polynomial $p$, then $\pm \l^{1/2}, \; \pm \l^{-1/2}$ are roots (with multiplicity one) of $p(x^2)$, and all other its roots have absolute value $1$. Let $q$ be the polynomial from the Definition \ref{def:square-rootable}. Then all roots of $q$ except one have absolute value at most $1$, and the remaining one has absolute value at most $R = Q^{1/2}$. Suppose that
		$$
		q(x) = x^{2m} + q_{2m-1} x^{2m-1} + \ldots + q_1 x + 1, \; q_k = q_{2m-k},
		$$
		then we can estimate its coefficients as
		\begin{equation}
			\label{eq:bounded_coefs}
			|q_{2m-k}| \le {2m \choose k} \cdot R \le cR. 
		\end{equation}
		On the other hand, all the coefficients of $q$ are algebraic integers because they can be polynomially expressed via roots of $p(x^2)$. 
		
		Multiplication of $\a$ by a square in $\QQ$ does not change the set $\sqrt{\a} \QQ$, thus we may assume $\a$ to be a square-free positive integer. So, even degree coefficients of $q$ lie in $\QQ \cap \AI = \ZZ$, and odd degree ones lie in
		$$
		\sqrt{\a} \QQ \cap \AI = \{ \sqrt{\a} \cdot r \mid r \in \QQ, \; \sqrt{\a} \cdot r \in \AI\} = \{ \sqrt{\a} \cdot r \mid r \in \QQ, \; \a r^2 \in \AI\} = \sqrt{\a} \ZZ.
		$$
		In particular, by inequality \eqref{eq:bounded_coefs} there is no reason to consider $\sqrt{\a} > c R$, since it would force all the odd degree coefficients to be zero.
		
		Hence there are at most $O(R)$ possible options for each of even degree coefficients of $q$ and, if we fix a square-free integer $\a \in [1, (cR)^2]$,  $O(R / \sqrt{\a})$ options for each of odd degree coefficients. Summing it up, the total number of possible polynomials $q(x)$ is at most
		%		$$
		%			C \cdot R^{[m/2]} \cdot \sum^{(cR)^2}_{\a = 1} \big(\frac{R}{\sqrt{\a}}\big)^{[(m+1)/2]} = C \cdot R^{[m/2] + [(m+1)/2]} \cdot \sum^{(cR)^2}_{\a = 1} \frac{1}{\sqrt{\a}^{[\frac{m+1}{2}]}} = O(R^m \cdot \eta(R^2, m)).
		%		$$
		\begin{equation}
			\label{eq:num_q(x)}
			C R^{[m/2]} \cdot \sum^{(cR)^2}_{\a = 1} \big(\frac{R}{\sqrt{\a}}\big)^{[(m+1)/2]} = C R^{[m/2] + [(m+1)/2]} \cdot \sum^{(cR)^2}_{\a = 1} \frac{1}{\sqrt{\a}^{[\frac{m+1}{2}]}} = C R^m \cdot \sum^{(cR)^2}_{\a = 1} \frac{1}{\sqrt{\a}^{[\frac{m+1}{2}]}}.
		\end{equation}
		The way we defined $\eta$ allows us to bound the sum $\sum\limits_{\a = 1}^{(cR)^2} \frac{1}{\sqrt{\a}^{[\frac{m+1}{2}]}}$ in the last expression as $O(\eta(R^2, m))$: 
		\begin{itemize}
			\item if $m = 1, 2$, then it equals $\sum\limits_{\a = 1}^{(cR)^2} \frac{1}{\a^{1/2}} \le 2\sqrt{(cR)^2} = O(R)$;
			\item if $m = 3, 4$, then it equals $\sum\limits_{\a = 1}^{(cR)^2} \frac{1}{\a} = O(\log R)$;
			\item if $m \ge 5$, then this sum converges, and is at most $\sum\limits_{\a = 1}^{(cR)^2} \frac{1}{\a^{3/2}} \le \zeta(\frac32) = O(1)$. 
		\end{itemize}
		So, we can continue equation \eqref{eq:num_q(x)} by concluding that there are $O(R^m \cdot \eta(R^2, m))$ possible polynomials $q(x)$. It remains to recall that $R = Q^{1/2}$ to finish the proof.
	\end{proof}
	
	Finally, we prove Theorem \ref{thm:multiplicities_alg}, which implies Theorem \ref{thm:multiplicities_geom} as explained in Section \ref{sec:bijection}. In order to do that, we combine Propositions \ref{prop:salem_all} and \ref{prop:salem_square-rootable} with Theorem \ref{thm:ERT} of Emery-Ratcliffe-Tschantz and with (a modification of) the Prime Geodesic Theorem of Margulis.
	
	\begin{proof}[\textbf{Proof of Theorem \ref{thm:multiplicities_alg}}]
		By the result of Gangolli and Warner \cite[Proposition 5.4]{GW80}, generalizing the classical theorem of Margulis \cite{Margulis69} to the setting of non-cocompact lattices, we can estimate the number of primitive hyperbolic conjugacy classes with translation lengths up to $L$ in a \textit{fixed} lattice $\Gamma$ as
		\[
		\# \Gamma_h(L) \sim \frac{e^{(n-1)L}}{(n-1)L}.
		\]
		Then for sufficiently large $L \ge L_0(n, \Gamma)$ 
		\begin{equation}	
			\label{eq:Margulis}
			\# \Gamma_h(L) \ge \frac{1}{2} \cdot \frac{e^{(n-1)L}}{(n-1)L}.
		\end{equation}
		
		\textit{Case of even $n$:} Theorem \ref{thm:ERT} asserts that the exponents $e^{\ell(H)}$ of translation lengths of hyperbolic transformations $H \in \Gamma$ are Salem numbers of degree at most $n+1$ (over $\QQ$, by Proposition \ref{prop:K=Q}). In fact, the degree is at most $n$ since degrees of Salem numbers are always even. By Proposition \ref{prop:salem_all} we have $\#\Sal_m(e^L) = O(e^{mL})$, and therefore,
		$$
		\#\{\ell(H) : H \in \Gamma_h(L)\} \le \sum^{n/2}_{m = 1}\#\Sal_m(e^L) = O(e^{nL/2}).
		$$
		Combining with formula \eqref{eq:Margulis} gives
		$$
		\frac{\# \Gamma_h(L)}{\# \{\ell(H) : H \in \Gamma_h(L)\}} \ge c(n) \frac{e^{(n/2 - 1)L}}{L}.
		$$
		
		\textit{Case of odd $n$:} Again, by Theorem \ref{thm:ERT} and Proposition \ref{prop:K=Q} the exponents $e^{2\ell(H)}$ are square-rootable over $\QQ$ Salem numbers of degree at most $n+1$. By Proposition \ref{prop:salem_square-rootable} we have 
		\[
		\#\Sal_m^{sq}(e^{2L}) = O(e^{mL} \cdot \eta(e^{2L}, m)),
		\]
		and therefore,
		\[
		\#\{\ell(H) : H \in \Gamma_h(L)\} \le \sum_{m = 1}^{(n+1)/2}\#\Sal^{sq}_m(e^{2L}) = O\Big(e^{(n+1)L / 2} \cdot \eta(e^{2L}, (n+1)/2)\Big).
		\]
		Since in this case $n$ is an odd integer at least $5$, the definition of $\eta$ yields
		$$
		\eta(e^{2L}, (n+1)/2) = \begin{cases} O(L), & \text{ if } n = 5, 7; \\ O(1), & \text{ otherwise,}\end{cases}
		$$
		which is exactly $O(L^{\delta_{5, 7}(n)})$.
		%		Since in this case $(n+1)/2 \ge 3$, the definition of $\eta$ yields
		%		$$
		%			\eta(e^{2L}, (n+1)/2) = O(L^{\delta_{5, 7}(n)}).
		%		$$
		Combining it with formula \eqref{eq:Margulis} gives
		$$
		\frac{\# \Gamma_h(L)}{\# \{\ell(H) : H \in \Gamma_h(L)\}} \ge c(n) \frac{e^{((n-1)/2 - 1)L}}{L^{1 + \delta_{5, 7}(n)}}.
		$$
	\end{proof}	
	
	\begin{remark}
		The argument we used to prove Theorem \ref{thm:multiplicities_alg} is quite straightforward. First we use the Prime Geodesic Theorem to count the number of primitive hyperbolic conjugacy classes in $\Gamma$ up to given translation length. Then we recall that any translation length in this situation can be expressed as logarithm of a (square-rootable) Salem number. Finally, we estimate the number of (square-rootable) Salem numbers of given degree on the initial segment of the real line. 
		
		The last part here is not perfectly linked to the first two, what gives some room for improvement: though it is true that any (square-rootable) Salem number arises from a translation length of a hyperbolic element in \textit{some} arithmetic lattice $\Gamma$ of the simplest type (see \cite[Theorems 1.1 and 1.6]{ERT15}), it seems unlikely that all the (square-rootable) Salem numbers can be generated by a \textit{single} lattice. 
		
		Indeed, analogous bounds for mean multiplicities in the case  $n = 3$, obtained by Marklof in \cite{Marklof96} by more careful analysis, are stronger than what is possible to get using our approach. Unfortunately, his methods are very specific to $3$-dimensional case. 		
	\end{remark}
	
	\begin{remark}
		\label{remark:explicit_constant}
		The proof of Theorem \ref{thm:multiplicities_alg} we presented does not intend to optimize the value of the constant factor $c(n)$ arising from Proposition \ref{prop:salem_all} for even dimensions and Proposition \ref{prop:salem_square-rootable} for odd dimensions. However, it is possible (and straightforward) to write down an explicit expression for it if one uses result of G\"otze and Gusakova (Theorem \ref{GG_1.1}) for even dimensions and our Theorem \ref{thm:asymptotics} for odd dimensions instead. More precisely, for fixed $n \ge 4$, fixed non-cocompact arithmetic lattice $\Gamma \subset \Isom(\HH^n)$, and $L \to \infty$
		\[
		\frac{\# \Gamma_h(L)}{\# \{\ell(H) : H \in \Gamma_h(L)\}} \ge \big( c'(n) + o_{n, \Gamma}(1) \big) \frac{e^{([n/2] - 1)L}}{L^{1 + \delta_{5, 7}(n) }},
		\] 
		\[
		\text{ where } \;\; c'(n) = \frac{1}{(n-1) \cdot w_{[\frac{n+1}{2}] - 1}} \cdot \begin{cases}  1, & \text{ if $n$ is even;} \\ \frac{\pi^2}{6}, & \text{ if $n = 5, 7$;} \\ \frac{\zeta([\frac{n+3}{4}])}{\zeta(\frac{1}{2}[\frac{n+3}{4}])}, & \text{ if $n \ge 9$ is odd.}\end{cases},
		\]
		and $w_m$ are defined by formula \eqref{eq:w_m}.
		Moreover, one can show that
		\[
		w_m \le \frac{4^{m}}{\sqrt{(m+1)!}}, 
		\]
		and therefore $c'(n)$ tends to infinity superexponentially with respect to $n$.  
	\end{remark}
	
	\section{Growth rate of square-rootable Salem numbers}
	\label{sec:GG-modification}
	
	In this section we adapt the argument of G\"otze and Gusakova \cite{GG19} to prove Theorem \ref{thm:asymptotics} and conclude that the bound from Proposition \ref{prop:salem_square-rootable} is essentially optimal.
	
	Before heading to the proof, we outline the strategy used by G\"otze and Gusakova to obtain the asymptotic formula for the number of Salem numbers of degree $2m$ lying in the interval $(1, Q]$ (all of them, not only square-rootable ones).
	\begin{theorem}[{\cite[Theorem 1.1]{GG19}}]
		For fixed $m$ and $Q \to \infty$ 
		\label{GG_1.1}
		$$
		\#\Sal_m(Q) = w_{m-1} Q^m + O(Q^{m-1}), 
		$$
		where $w_m$ are defined by the same formula \eqref{eq:w_m}.
		%		\begin{equation}
			%		\label{eq:w_m}
			%			w_m = \frac{2^{m(m+1)}}{m+1}\prod_{k = 0}^{m-1} \frac{(k!)^2}{(2k+1)!}.
			%		\end{equation}
	\end{theorem}
	Let $V^{all}_m$ be the set of monic palindromic polynomials in $\RR[x]$ of degree $2m$, identified with a subset of $\RR^m$ via considering their coefficients at $x, x^2, \ldots, x^m$. Define an injective map 
	\[
	f_Q: \{(\l, \theta_1, \ldots, \theta_{m-1}) \in \RR^m \mid 1 < \l \le Q, \; 0 \le \theta_1 \le \ldots, \le \theta_{m-1} \le \pi\} \to V^{all}_m \; \text{ as }
	\]	
	\begin{equation}
		\label{eq:f}
		(\l, \theta_1, \ldots, \theta_{m-1}) \mapsto (x-\l)(x-\l^{-1})\prod_{j = 1}^{m-1}(x-e^{i\theta_j})(x-e^{-i\theta_j}).
	\end{equation}
	Then each integer point inside $V_m(Q) = \im f_Q \subset \RR^m$ corresponds to some integer polynomial, with one of its roots being a Salem number. Since the number of reducible integer polynomials inside $V_m(Q)$ is $O(Q^{m-1})$ (see \cite[Lemma 2.2]{GG19}), most of integer points are exactly minimal polynomials of Salem numbers.
	
	To have a chance to count the number of integer points inside, one needs to check that the boundary of $V_m(Q)$ is ``nice enough''.
	
	\begin{definition}
		We say that the boundary $\partial B$ of a set $B \subset \RR^{m}$ is of Lipschitz class $(M, L)$ if there exist $M$ maps
		\[
		\varphi_1, \ldots, \varphi_M : [0, 1]^{m-1} \to \RR^{m}, 
		\]
		such that each of them is $L$-Lipschitz
		\[
		\lVert\varphi_i(x) - \varphi_i(y)\rVert \le L \lVert x-y \rVert \text{ for any } x, y \in [0, 1]^{m-1}, \; 1 \le i \le M
		\]
		and $\partial B$ is covered by the union of their images.
	\end{definition}
	
	\begin{lemma}[{\cite[Lemma 2.1]{GG19}}]
		\label{lemma:M_cQ}
		Boundary of $V_m(Q)$ is of Lipschitz class $(M, cQ)$ for some constants $M, c > 0$.
	\end{lemma}
	Then lattice point counting result of Widmer allows to estimate the number of lattice points inside $V_m(Q)$ by its volume.	
	\begin{theorem}[{\cite[Theorem 5.4]{Widmer10}}]
		\label{thm:lattice_points}
		Let $\Lambda$ be a lattice in $\RR^d$ with successive minima $\l_1, \ldots, \l_d$, and let $B \subset \RR^d$ be a bounded set with boundary of Lipschitz class $(M, L)$. Then $B$ is measurable, and the number $\#(\Lambda \cap B)$ of lattice points inside $B$ satisfies
		$$
		\Bigl| \#(\Lambda \cap B) - \frac{\Vol(B)}{\det \Lambda} \Bigr| \le c(d) M \max_{0 \le i < d} \frac{L^i}{\l_1 \ldots \l_i}.
		$$
	\end{theorem}
	So, the problem is reduced to computation of the volume $v_m(Q) = \Vol(V_m(Q))$. Careful calculation of Jacobian of the map $f_Q$ (see \cite[Lemma 2.4]{GG19}) and its integration over the domain provide the estimate
	\begin{equation}
		\label{eq:v_m_w_m}
		v_m(Q) = w_{m-1} Q^m + O(Q^{m-1}), 
	\end{equation}
	thus finishing the proof of Theorem \ref{GG_1.1}.
	
	\bigskip
	
	Now return to our Theorem \ref{thm:asymptotics}. We want to count the number of square-rootable (over $\QQ$) Salem numbers of degree $2m$ up to $Q$, each of them has an associated polynomial $q(x)$ from Definition \ref{def:square-rootable}. The following Definition \ref{def:P_m} specifies which exactly polynomials we are interested in.
	
	\begin{definition}
		\label{def:P_m}
		Denote by $P_m$ the set of pairs $(\a, q)$, where $\a$ is a positive square-free integer and $q$ is a polynomial in $\RR[x]$ such that
		\begin{enumerate}[label=(\alph*)]
			\item $q$ is monic and palindromic;
			\item $\deg q = 2m$;
			\item even degree coefficients of $q$ lie in $\ZZ$, odd degree coefficients of $q$ lie in $\sqrt{\a} \ZZ$;
			\item all the roots of $q$ except two have absolute value exactly $1$, and the remaining two are positive reals $\l > 1$ and $\l^{-1} < 1$ (both with multiplicity one);
			\item $q(\xi) \neq 0$ for any root of unity $\xi$ of degree at most $4m$.
		\end{enumerate}
		Also define a map $\Phi: P_m \to \RR$ by
		\[
		\Phi: (\a, q) \mapsto \l^2, \text{ where } \l \text{ is the unique root of $q$ with absolute value greater than $1$}.
		\]
	\end{definition}
	
	\medskip
	
	The following lemma resembles Definition \ref{def:square-rootable}, and will be used to relate the set $P_m$ to square-rootable Salem numbers.
	
	\begin{lemma}
		\label{lemma:q(x)q(-x)}
		For any $(\a, q) \in P_m$ we have 
		\[
		q(x) q(-x) = p(x^2),
		\]
		where $p \in \ZZ[x]$ is the minimal polynomial of $\l^2 = \Phi((\a, q))$ over $\QQ$.
	\end{lemma}
	\begin{proof}		
		Consider the polynomial
		\[
		r(x) = q(x) q(-x).
		\]		
		Notice that $r(x) = r(-x)$, therefore all the odd degree coefficients of $r$ are zero. Moreover, conditions (b) and (c) from Definition \ref{def:P_m} imply that even degree coefficients of $r$ are integers. Indeed, one can write		
		\[
		q(x) = a_{2m} x^{2m} + \sqrt{\a} a_{2m-1} x^{2m-1} + a_{2m-2} x^{2m-2} + \ldots \; \text{ with } a_i \in \ZZ.
		\]
		Then the coefficient of $r(x)$ at $x^{2k}$ equals 
		\[
		\sum_{l = 0}^{k} a_{2l} a_{2(k - l)} - \sum_{l = 0}^{k-1} \sqrt{\a}a_{2l+1} \cdot \sqrt{\a} a_{2(k-l)-1} = \sum_{l = 0}^{k} a_{2l} a_{2(k - l)} - \sum_{l = 0}^{k-1} \a a_{2l+1} a_{2(k-l)-1}, 
		\]
		which is clearly an integer. So, $r(x) = p(x^2)$ for some monic polynomial $p \in \ZZ[x]$.
		
		\medskip
		
		If $\l$ is the root of $q$ from condition (d), then $\l^2 > 1,\; \l^{-2} < 1$ are roots (with multiplicity one) of $p$, and all other its roots have absolute value $1$. To check that $p$ is the minimal polynomial of $\l^2$, it remains to verify its irreducibility. 
		
		Suppose that it is reducible: $p(x) = p_1(x) p_2(x)$ for some monic polynomials $p_1, p_2 \in \ZZ[x]$, and $\l^{-2}$ is a root of $p_1$. Observe that if $\l^2$ is not a root of $p_1$, then the product of its roots has absolute value $\l^{-2} \in (0, 1)$, but its constant term is an integer. 
		So, $\l^2$ is a root of $p_1$ as well. Thus $p_2$ has only roots of absolute value $1$, and the classical result of Kronecker \cite{Kronecker1857} forces them to be roots of unity. Therefore, $p(\xi_1) = 0$ for some root of unity $\xi_1$ of degree at most $2m$, and then $q(\xi_2) = 0$ for some root of unity $\xi_2$ of degree at most $4m$, what contradicts condition (e).
	\end{proof}
	
	Next we observe that inside the set $P_m$ polynomial $q$ determines the value of $\a$ uniquely.
	
	\begin{lemma}
		\label{lemma:q_determines_alpha}
		If $(\a, q) \in P_m$ then at least one odd degree coefficient of $q$ is non-zero. As a consequence, if $(\a_1, q), \; (\a_2, q) \in P_m$, then $\a_1 = \a_2$.
	\end{lemma}
	\begin{proof}
		Suppose that all the odd degree coefficients of $q$ are zero, then $q(x) = q(-x)$. In particular, if $\l$ is a root of $q$, then so is $-\l$, contradiction with condition (d). 
		
		For the second part, if $\a_1, \a_2$ are distinct positive square-free integers, then all the odd degree coefficients of $q$ have to lie in $\sqrt{\a_1} \ZZ \cap \sqrt{\a_2} \ZZ = \{0\}$. Contradiction with the first part of this lemma.  
	\end{proof}
	
	The next two lemmas assert that the map $\Phi$ from Definition \ref{def:P_m} is almost a bijection onto the set of square-rootable (over $\QQ$) Salem numbers of degree $2m$.
	
	\begin{lemma}
		\label{lemma:im=sal}
		The image of $\Phi$ coincides with $\Sal_m^{sq}$.
	\end{lemma}
	
	\begin{lemma}
		\label{lemma:size_of_preimage}
		For any $x \in \RR$
		\[
		\#(\Phi^{-1}(x)) \le C(m) = 2^{2m}. 
		\]
		Moreover, if $m$ is odd then for any $x \in \RR$ 
		\[
		\#(\Phi^{-1}(x)) \le 1. 		
		\]
		In other words, if $m$ is odd then $\Phi$ is injective.
	\end{lemma}
	
	\begin{proof}[\textbf{Proof of Lemma \ref{lemma:im=sal}}]		
		First we check that the image of $\Phi$ lies inside $\Sal_m^{sq}$. 
		Given $(\a, q) \in P_m$, Lemma \ref{lemma:q(x)q(-x)} tells us that
		\begin{equation}
			\label{eq:from_q(x)q(-x)}
			q(x) q(-x) = p(x^2),
		\end{equation}
		where $p$ is a monic polynomial with integer coefficients, which is also the minimal polynomial of $\l^2 = \Phi((\a, q))$.	
		Again, by condition (d), the roots of $p$ are $\l^2 > 1,\; \l^{-2} < 1$ (both with multiplicity one) and other ones of absolute value $1$.		
		So, $\l^2 = \Phi((\a, q))$ is indeed a Salem number of degree $\deg p = \deg q = 2m$. Then equation \eqref{eq:from_q(x)q(-x)} combined with condition (c) implies that it is also square-rootable.
		
		\medskip
		
		Now we check that $\Sal^{sq}_m$ is contained in the image of $\Phi$. Suppose $\l$ is a square-rootable Salem number with minimal polynomial $p$, $\deg p = 2m$. Consider $\a \in \QQ$ and $q \in \RR[x]$ from Definition \ref{def:square-rootable}. Multiplication of $\a$ by a square in $\QQ$ does not change the set $\sqrt{\a} \QQ$, therefore we may assume that $\a$ is indeed a square-free positive integer. The polynomial $q$ satisfies conditions (a) and (b) automatically. Since
		\[
		q(x)q(-x) = p(x^2), 
		\] 
		then $\pm \l^{1/2}, \pm \l^{-1/2}$ are roots (with multiplicity one) of $q(x)q(-x)$, and all other its roots have absolute value $1$. Recall that a palindromic polynomial has root $z$ if and only if it has root $z^{-1}$, what allows us to conclude (by replacing $q(x)$ with $q(-x)$ if necessary) that $q$ satisfies condition (d) as well. 
		
		Condition (e) follows from irreducibility of $p$, and to check condition (c) we repeat the argument used in the proof of Proposition \ref{prop:salem_square-rootable}: since coefficients of $q$ can be polynomially expressed via its roots, they are algebraic integers; thus, even degree coefficients of $q$ lie in $\QQ \cap \AI = \ZZ$, odd degree coefficients of $q$ lie in 		
		\[
		\sqrt{\a} \QQ \cap \AI = \{ \sqrt{\a} \cdot r \mid r \in \QQ, \; \sqrt{\a} \cdot r \in \AI\} = \{ \sqrt{\a} \cdot r \mid r \in \QQ, \; \a r^2 \in \AI\} = \sqrt{\a} \ZZ.
		\]		
		So, we constructed $(\a, q) \in P_m$, such that $\Phi((\a, q)) = \l$.		
	\end{proof}
	
	\medskip
	
	\begin{proof}[\textbf{Proof of Lemma \ref{lemma:size_of_preimage}}]
		If $p$ is the minimal polynomial of a square-rootable Salem number $\l^2 = \Phi((\a, q))$, then by Lemma \ref{lemma:q(x)q(-x)} we have $q(x) q(-x) = p(x^2)$. Since $q$ determines the value of $\a$ by Lemma \ref{lemma:q_determines_alpha}, the size of $\Phi^{-1}(\l^2)$ is bounded above by the number of decompositions of $p(x^2)$ into a product of the form $q(x)q(-x)$. But any such decomposition is given by choosing one representative from each pair of opposite roots of $p(x^2)$, thus there are at most $2^{2m}$ of them.
		
		To prove injectivity of $\Phi$ in the case of odd $m$ more careful analysis is required. Fix a square-rootable Salem number $\l^2$ of degree $2m$ in the image of $\Phi$.
		Suppose that $\Phi((\a, q)) = \l^2$ for some $(\a, q) \in P_m, \; \a > 1$, then $\l$ is the unique root of $q$ with absolute value greater than one. Conditions (b) and (c) allow us to write
		\[
		q(x) = a_{2m} x^{2m} + \sqrt{\a} a_{2m-1} x^{2m-1} + a_{2m-2} x^{2m-2} + \sqrt{\a} a_{2m-3} x^{2m-3} + \ldots, \; \text{ with } a_i \in \ZZ.
		\]
		Substituting $\l$ and separating even and odd terms gives
		\[
		\sqrt{\a} = \frac{- a_{2m} \l^{2m} - a_{2m-2} \l^{2m-2} - \ldots}{a_{2m-1} \l^{2m-1} + a_{2m-3} \l^{2m-3} + \ldots} \in \QQ(\l).
		\]
		Note that denominator of this fraction is non-zero: indeed, $\deg_{\QQ} \l \ge \deg_{\QQ} \l^2 = 2m$, and we are not allowed to have all the odd degree coefficients equal to zero by Lemma \ref{lemma:q_determines_alpha}.
		
		\medskip
		\textbf{Case 1.} $\Phi((1, q_1)) = \l^2$ for some $(1, q_1) \in P_m$.		
		\medskip
		
		Then $q_1 \in \ZZ[x]$ is the minimal polynomial of $\l$ over $\QQ$, thus there is no other pair in $\Phi^{-1}(\l^2)$ with the first coordinate equal to $1$. Suppose that $\Phi((\a, q)) = \l^2$ for some $(\a, q) \in P_m, \; \a > 1$. From $[\QQ(\l) : \QQ] = \deg q_1 = 2m$ and $\QQ(\sqrt{\a}) \subset \QQ(\l)$ we deduce that
		\[
		[\QQ(\l) : \QQ(\sqrt{\a})] = m, \;\; \text{which is an odd number.}
		\]
		Let $r_1$ be the minimal polynomial of $\l$ over $\QQ(\sqrt{\a})$, $\deg r_1 = m$. Then $r_{-1}(x) = x^m r_1(x^{-1})$ is the minimal polynomial of $\l^{-1}$ over $\QQ(\sqrt{\a})$, and $\deg r_{-1} = m$ as well. Since they are minimal, and $q_1, q \in \QQ(\sqrt{\a})[x]$ have $\l$ and $\l^{-1}$ as their roots by condition (d), we have
		\[
		r_1, r_{-1} \mid q_1, \;\; r_1, r_{-1} \mid q. 
		\]
		Conditions (d) and (e) imply that $q$ has only two real roots. Since the degrees of $r_1$ and $r_{-1}$ are odd, $\l$ is the unique real root of $r_1$ and $\l^{-1}$ is the unique real root of $r_{-1}$. In particular, $r_1 \neq r_{-1}$. Since they are both irreducible over $\QQ(\sqrt{\a})$, this implies that they are coprime, and
		\[
		r_1 r_{-1} \mid q_1, \; r_1 r_{-1} \mid q. 
		\]
		From $\deg r_1 r_{-1} = \deg q = \deg q_1 = 2m$ we conclude that $q = q_1$. But $\a \neq 1$, which is a contradiction with Lemma \ref{lemma:q_determines_alpha}.
		
		\medskip
		
		\textbf{Case 2.} $\Phi((1, q_1)) \neq \l^2$ for any $(1, q_1) \in P_m$. Equivalently, $[\QQ(\l) : \QQ] \neq 2m$.
		\medskip
		
		Since $\l^2$ is a Salem number of degree $2m$, the degree of $\QQ(\l)$ over $\QQ$
		\[
		[\QQ(\l) : \QQ] = [\QQ(\l) : \QQ(\l^2)] \cdot [\QQ(\l^2) : \QQ]
		\] 
		is either $4m$ or $2m$. Since the latter option is ruled out by the assumption, only the first one remains. Suppose that
		\[
		\Phi((\a, q)) = \Phi((\a', q')) = \l^2 \; \text{ for some } (\a, q), (\a', q') \in P_m, \; \a, \a' > 1.
		\]
		This implies, as noted above, that $\sqrt{\a}, \sqrt{\a'} \in \QQ(\l)$, and
		\[
		[\QQ(\l) : \QQ(\sqrt{\a})] = [\QQ(\l) : \QQ(\sqrt{\a'})] = 2m.
		\] 
		Then $q \in \QQ(\sqrt{\a})[x]$ and $q' \in \QQ(\sqrt{\a'})[x]$ are the minimal polynomials of $\l$ over $\QQ(\sqrt{\a})$ and, respectively, $\QQ(\sqrt{\a'})$. Hence $\a = \a'$ would imply $q = q'$, and we may assume $\a \neq \a'$. 
		
		The remaining part of the argument is similar to the previous case. The degree
		\[
		[\QQ(\l) : \QQ(\sqrt{\a}, \sqrt{\a'})] = m 
		\] 
		is odd, thus the minimal polynomials $r_1$ and $r_{-1}$ of $\l$ and, respectively, $\l^{-1}$ over $\QQ(\sqrt{\a}, \sqrt{\a'})$ do not coincide, and therefore are coprime. By their minimality (and condition (d) for $q$ and $q'$) both $r_1$ and $r_{-1}$ divide $q$ and $q'$, hence so does their product $r_1 r_{-1}$. From
		\[
		\deg r_1 r_{-1} = \deg q = \deg q' = 2m
		\] 
		we deduce that $q = q'$. But $\a \neq \a'$, which is a contradiction with Lemma \ref{lemma:q_determines_alpha}.		
	\end{proof}
	
	\begin{remark}
		This sort of ambiguity for even $m$ in Lemma \ref{lemma:size_of_preimage} is the reason why we were not able to obtain the exact asymptotics in Theorem \ref{thm:asymptotics} for even $m$. Note that the map $\Phi$ actually can be non-injective: for example, the root of Salem polynomial $p(x) = x^8 - 56 x^7 - 157 x^6 - 228 x^5 - 247 x^4 - 228 x^3 - 157 x^2 - 56 x + 1$ has four preimages in $P_4$ with the values of $\a$ equal to $2, 6, 26$ and $78$, respectively. It seems unlikely for such phenomenon to occur ``too often'', but we were unable to make this statement rigorous.	
		
		%[1, -56, -157, -228, -247, -228, -157, -56, 1]			
		%58.740107819108246			
		%2, 6, 26, 78
	\end{remark}
	
	The following lemma will help us to handle condition (e) while counting polynomials in the set $P_m$ from Definition \ref{def:P_m}.
	\begin{lemma}
		\label{lemma:minus_planes}
		Suppose that the boundary $\partial B$ of some bounded set $B \subset \RR^{m}$ is of Lipschitz class $(M, L)$, and a set $H \subset \RR^{m}$ can be covered by a finite family of hyperplanes $H_1, \ldots, H_K$. Then the boundary of 
		$$
		B' = B \setminus H
		$$ 
		is of Lipschitz class $(MK+M, L)$.
	\end{lemma}
	\begin{proof}
		Consider the maps $\{\varphi_i\}_{1 \le i \le M}$ covering $\partial B$ by their images. Define $\{\varphi_{ij}\}_{1 \le i \le M, 1 \le j \le K}$ as the composition of $\varphi_i$ with orthogonal projection onto $H_j$. Since the projections do not increase distances, these maps are also $L$-Lipschitz. If $x \in \partial B' \setminus \partial B$, then $x$ lies in the interior of $B$ intersected with one of the hyperplanes $H_{j_0}$. Since $B$ is bounded, the line orthogonal to $H_{j_0}$ passing through $x$ has to intersect the boundary $\partial B$. Therefore $\partial B'$ is covered by the images of $\{\varphi_i\}$ and $\{\varphi_{ij}\}$.
	\end{proof}	
	
	Also we record two simple facts about sums of powers of square-free integers (see \cite{Scott06}, \cite{Hirschhorn03}, \cite[OEIS/A005117]{OEIS}), and include their proofs for reader's convenience. 
	
	\begin{proposition}
		\label{prop:sq-free_log}		
		For $x \to \infty$,
		\[
		\sum_{n \in \NN, \; n \le x, \; n \text{ is square-free}} \frac{1}{n} \sim \frac{6}{\pi^2} \log(x).
		\]
	\end{proposition}
	
	\begin{proof}
		We need to check that the sequence of real numbers
		\[
		s_N = \frac{\pi^2}{6} \cdot \frac{1}{\log N} \cdot \sum_{n \in \NN, \; n \le N, \; n \text{ is square-free}} \frac{1}{n} \qquad \qquad \qquad (N \in \NN, \; N \ge 2)
		\]
		tends to $1$ as $N \to \infty$. Each natural number can be uniquely represented as a product of a square-free natural number and a square of a natural number. Therefore, 
		\[
		s_N = \Big(\sum_{n \in \NN} \frac{1}{n^2} \Big) \cdot \frac{1}{\log N} \cdot \Big(\sum_{n \in \NN, \; n \le N, \; n \text{ is square-free}} \frac{1}{n} \Big) \ge \frac{1}{\log N} \cdot \sum_{n \le N} \frac{1}{n} \ge 1.
		\]		
		On the other hand, for any $\eps > 0$ we can fix $k \in \NN$, such that
		\[
		\frac{\pi^2}{6} \le (1 + \eps) \sum_{n \le k} \frac{1}{n^2}.
		\]
		Then 
		\begin{align*}
			s_N &\le (1 + \eps) \cdot \Big(\sum_{n \le k} \frac{1}{n^2} \Big) \cdot \frac{1}{\log N} \cdot \Big(\sum_{n \in \NN, \; n \le N, \; n \text{ is square-free}} \frac{1}{n} \Big) \\ &\le (1 + \eps) \cdot \frac{1}{\log N} \cdot \sum_{n \le k^2N} \frac{1}{n} \le (1 + \eps) \cdot \frac{\log N + 2\log k + 1}{\log N} \to 1 + \eps.
		\end{align*}
		So, for any $\eps > 0$
		\[
		1 \le \liminf_{N \to \infty} s_N \le \limsup_{N \to \infty} s_N \le 1 + \eps,
		\]
		and we send $\eps \to 0$ to finish the proof.
	\end{proof}
	
	\begin{proposition}
		\label{prop:sq-free_zeta}
		For $s > 1$,
		\[
		\sum_{n \in \NN, \; n \text{ is square-free}} n^{-s} = \frac{\zeta(s)}{\zeta(2s)}.
		\]
	\end{proposition}
	
	\begin{proof}
		Again, recall that each natural number can be uniquely represented as a product of a square-free natural number and a square of a natural number. Therefore, 
		\[
		\Big(\sum_{n \in \NN} \frac{1}{n^{2s}}\Big) \cdot \Big( \sum_{n \in \NN, \; n \text{ is square-free}} \frac{1}{n^s} \Big) = \sum_{n \in \NN} \frac{1}{n^s}.
		\]
		Since $s > 1$, each of these three series absolutely converges. Then by definition of the Riemann zeta function we conclude that
		\[
		\sum_{n \in \NN, \; n \text{ is square-free}} \frac{1}{n^s} = \Big(\sum_{n \in \NN} \frac{1}{n^s}\Big) \Big/ \Big(\sum_{n \in \NN} \frac{1}{n^{2s}}\Big) = \frac{\zeta(s)}{\zeta(2s)}.
		\]
	\end{proof}
	
	\begin{proof}[\textbf{Proof of Theorem \ref{thm:asymptotics}}]
		For a positive square-free integer $\a$ define
		\[
		P_{m ,\a} = P_{m, \a}(Q) = \{q \in \RR[x] : (\a, q) \in P_m, \; \Phi((\a, q))  \le Q\}, 
		\]
		where $P_m$ and $\Phi$ are as in Definition \ref{def:P_m}. These sets are disjoint by Lemma \ref{lemma:q_determines_alpha}, and
		by Lemma \ref{lemma:im=sal} and Lemma \ref{lemma:size_of_preimage} we have
		\begin{equation}
			\label{eq:odd}
			\#\Sal_m^{sq}(Q) = \sum_{\a \in \NN, \; \text{square-free}} \# P_{m, \a} \;\;\; \text{ for odd $m$, and }
		\end{equation}
		\begin{equation}
			\label{eq:even}
			\frac{1}{2^{2m}}\Big(\sum_{\a \in \NN, \; \text{square-free}} \# P_{m, \a} \Big) \le \# \Sal_m^{sq}(Q) \le \sum_{\a \in \NN, \; \text{square-free}} \# P_{m, \a}  \;\;\; \text{ for even $m$.}
		\end{equation}
		
		Notice that these sums are actually finite.  Indeed, similarly to the argument from the proof of Proposition \ref{prop:salem_square-rootable}, we see that $q \in P_{m, \a}$ has exactly one root with absolute value greater than $1$, and its absolute value is at most $R = Q^{1/2}$. Therefore, absolute values of coefficients of $q$ have to be bounded above by $c_0 R$ for some constant $c_0 = c_0(m)$, and we can not have $\a > (c_0 R)^2$.
		
		So, the problem is reduced to calculation of the sizes of sets $P_{m, \a}$.
		
		\medskip
		
		Following the strategy of G\"otze and Gusakova, let $V^{all}_m$ be the set of monic palindromic polynomials in $\RR[x]$ of degree $2m$, identified with a subset of $\RR^m$ via considering the coefficients at $x, x^2, \ldots, x^m$. 
		Consider the map $f_R$ defined as in formula \eqref{eq:f} 
		\[
		f_R: \{(\l, \theta_1, \ldots, \theta_{m-1}) \in \RR^m \mid 1 < \l \le R, \; 0 \le \theta_1 \le \ldots, \le \theta_{m-1} \le \pi\} \to V^{all}_m,
		\]			
		\[
		(\l, \theta_1, \ldots, \theta_{m-1}) \mapsto (x-\l)(x-\l^{-1})\prod_{j = 1}^{m-1}(x-e^{i\theta_j})(x-e^{-i\theta_j}).
		\]	
		
		Let $V_m(R) \subset V_m^{all}$ be the image of $f_R$. Condition (c) suggests us to look at the lattice $\Lambda_\a = \sqrt{\a}\ZZ e_1 + \ZZ e_2 + \sqrt{\a}\ZZ e_3 + \ldots \subset \RR^m$. 
		Observe that the set of lattice points $\Lambda_\a \cap V_m(R)$ consists of polynomials satisfying conditions (a)-(d) from Definition \ref{def:P_m}, with the root $\l$ being at most $R = Q^{1/2}$. Therefore, this set almost coincides with $P_{m, \a}$, and to make it satisfy condition (e) as well, we exclude from $V_m(R)$ the following sets:
		\[
		Z(\xi) = \{ q \in V^{all}_m \mid q(\xi) = 0\}, \text{ where $\xi$ is a root of unity of degree at most $4m$}.
		\]
		Both $\Ree (q(\xi)) = 0$ and $\Imm(q(\xi)) = 0$ are linear equations on the coefficients of $q$, and at least one of them is non-trivial --- therefore, each $Z(\xi)$ lies inside a hyperplane. Recall that boundary of $V_m(R)$ is of Lipschitz class $(M, cR)$ by Lemma \ref{lemma:M_cQ}, then Lemma \ref{lemma:minus_planes} allows to conclude that the boundary of 
		$$
		V'_m(R) = V_m(R) \setminus \Big(\bigcup_{\xi^{d} = 1, \; d \le 4m} Z(\xi)\Big)
		$$
		is of Lipschitz class $(M', cR)$ for some constant $M' \le (4m)^2 M$. Now we indeed have $P_{m, \a} = \Lambda_\a \cap V'_m(R)$, and lattice points counting Theorem \ref{thm:lattice_points} yields
		$$
		\Bigl| \# P_{m, \a} - \frac{\Vol(V'_m(R))}{\sqrt{\a}^{[\frac{m+1}{2}]}} \Bigr| = \Bigl| \#(\Lambda_\a \cap V'_m(R)) - \frac{\Vol(V'_m(R))}{\sqrt{\a}^{[\frac{m+1}{2}]}} \Bigr| = O\Big( \frac{R^{m-1}}{\sqrt{\a}^{[\frac{m+1}{2}]-1}}\Big) \; \text{ for } \a \le (c_0 R)^2.
		$$
		Recall the volume estimate $\eqref{eq:v_m_w_m}$
		$$
		\Vol(V'_m(R)) = \Vol(V_m(R)) = w_{m-1} R^m + O(R^{m-1}), 
		$$
		then
		$$
		\# P_{m, \a} = w_{m-1} \frac{R^m}{\sqrt{\a}^{[\frac{m+1}{2}]}} +  O\Big(\frac{R^{m-1}}{\sqrt{\a}^{[\frac{m+1}{2}]-1}}\Big).
		$$
		Taking the sum over square-free positive integers up to $(c_0 R)^2$ gives
		\begin{equation}
			\label{eq:sum_P_a}
			\sum_{\substack{1 \le \a \le (c_0 R)^2, \\ \text{square-free}}} \# P_{m, \a} = w_{m-1} R^m \cdot \Big( \sum_{\substack{1 \le \a \le (c_0 R)^2, \\ \text{square-free}}} \frac{1}{\sqrt{\a}^{[\frac{m+1}{2}]}} \Big)  + O\Big(R^{m-1} \cdot \sum_{\substack{1 \le \a \le (c_0 R)^2, \\ \text{square-free}}} \frac{1}{\sqrt{\a}^{[\frac{m+1}{2}]-1}} \Big)
		\end{equation}
		Next we apply Proposition \ref{prop:sq-free_log} to compute the first term of equation \eqref{eq:sum_P_a} for $m = 3, 4$ (and use a trivial bound $\sum_{n \le x} 1/\sqrt{n} = O(\sqrt{x})$ for the second term):
		\begin{equation}
			\label{eq:34}
			\sum_{\substack{1 \le \a \le (c_0 R)^2, \\ \a \text{ is square-free}}} \# P_{m, \a} = w_{m-1} R^m \cdot \frac{6}{\pi^2} \cdot 2 \log(R) \cdot (1 + o(1)) + O(R^m).  
		\end{equation}
		Similarly, we use Proposition \ref{prop:sq-free_zeta} to compute the first term of equation \eqref{eq:sum_P_a} for $m \ge 5$:		
		\begin{equation}
			\label{eq:ge5}
			\sum_{\substack{1 \le \a \le (c_0 R)^2, \\ \a \text{ is square-free}}} \# P_{m, \a} = w_{m-1} R^m \cdot \frac{\zeta(\frac12 [\frac{m+1}{2}])}{\zeta([\frac{m+1}{2}])} \cdot (1 + o(1)) + O(R^{m-1} \log R).
		\end{equation}
		It remains to recall that $R = Q^{1/2}$ and combine equations \eqref{eq:34} and \eqref{eq:ge5} with equations \eqref{eq:odd} and \eqref{eq:even} to obtain the desired estimates.
		
	\end{proof}
	
	\bibliographystyle{plain}
	
	\bibliography{Salem_numbers}
	
\end{document}